\documentclass[12pt,a4paper]{scrartcl}

\usepackage{mystyle}
\title{Intersections of Poisson \texorpdfstring{$k$}{k}-flats in hyperbolic space: completing the picture}

\begin{document}
\author{Tillmann Bühler\footnotemark[1]~ and Daniel Hug\footnotemark[2]}
\date{}
\maketitle
\renewcommand{\thefootnote}{\fnsymbol{footnote}}

\footnotetext[1]{Institute of Stochastics, Karlsruhe Institute of Technology, tillmann.buehler@kit.edu}
\footnotetext[2]{Institute of Stochastics, Karlsruhe Institute of Technology, daniel.hug@kit.edu}

\begin{abstract}
    In recent years there has been a lot of interest in the study of isometry invariant Poisson processes
	of \(k\)-flats in $d$-dimensional hyperbolic space $\mathbb{H}^d$, for  $0\le k\le d-1$.
	A phenomenon that has no counterpart in euclidean geometry arises in the investigation of the total \(k\)-dimensional volume
	\(F_r\) of the process inside a spherical observation window \(B_r\) of radius $r$ when one lets \(r\) tend to infinity.
	While \(F_r\) is asymptotically normally distributed for \(2k\leq d+1\), it has been shown to obey a nonstandard central limit theorem
	for \(2k>d+1\).

	The intersection process of order $m$, for $d-m(d-k) \geq 0$, of the original process \(\eta\) consists of all intersections of distinct
	flats
	\(E_1,\ldots,E_m \in \eta\) with \(\dim(E_1\cap\ldots\cap E_m) = d-m(d-k)\).
	For this intersection process, the total \(d-m(d-k)\)-dimensional volume \(F^{(m)}_r\) of the process in \(B_r\), again as \(r \to \infty\), is of particular interest.
	For \(2k \leq d+1\) it has been shown that \(F^{(m)}_r\) is again asymptotically normally distributed.
	For \(m \geq 2\), the limit is so far unknown, although it has been shown for certain \(d\) and \(k\) that it cannot be a normal distribution.

	We determine the limit distribution for all values of \(d,k,m\).
	In addition, we establish explicit rates of convergence in the Kolmogorov distance and discuss properties of the limit distribution.
    Furthermore we show that the asymptotic covariance matrix of the vector \((F^{(1)}_r,\ldots,F^{(m)}_r)^\top\) has full rank when \(2k < d+1\) and rank one
	when \(2k \geq d+1\).
\end{abstract}

\paragraph{\small MSC-classes 2020.}{\small Primary: 60D05, 53C65, 52A22, Secondary: 52A55, 60F05 }
\paragraph{\small Keywords.}{\small Poisson process, hyperbolic space, $k$-flat process,  limit theorem, Kolmogorov distance, Berry--Esseen bound}

	\section{Introduction}
\label{sec:intro}

Random geometric systems have been explored in euclidean space both from an applied and a theoretical point of view (see, e.g., the monographs
\cite{NET-006,NET-026,BHKM2018,MR1409145,CSKM,Coupier19,Haenggi13,MR973404,HugSch24,KendallMolchanov10,Matheron75,MR2283880,MR3753715,SW08}).
Recently, the investigation of random structures in hyperbolic space and other non-euclidean spaces has come into focus. In particular,
random graphs \cite{CFS22,FY20,hoot2023,rosen2024radialspanningtreehyperbolic}, percolation theory \cite{BS01,BJST09,HM22,hm2023}, 
random polytopes and geometric probabilities \cite{MR4753064,MR4333417,MR4177280,sönmez2024intersectionprobabilitiesflatshyperbolic,MR4177280,Kabluchko2020ANA}, 
random tessellations \cite{BPP18,GKT22,dac2023,Hug2018SplittingTI,Kabluchko2019TheTC}, 
random systems of hyperplanes \cite{HHT,BHT,KRT,KRT22+,Kabluchko2020FacesIR} and 
Boolean models \cite{T07,T09,TC13,hls2024} have been considered. 
In the present work,
we continue the exploration of the asymptotic fluctuations related to functionals of isometry invariant Poisson processes of $k$-flats ($k$-dimensional totally geodesic submanifolds) in $d$-dimensional hyperbolic space, which was initiated in \cite{HHT}.

To be more specific, take \(k \in \{1,\ldots,d-1\}\) and \(m\in \N\) so that \(d-m(d-k) \geq 0\) and let $\eta$ denote a stationary Poisson process of $k$-flats in hyperbolic space $\mathbb{H}^d$.
We consider the intersection process of order $m$ of the original Poisson process \(\eta\), which  consists of all intersections of distinct $k$-flats \(E_1,\ldots,E_m \in \eta\) with \(\dim(E_1\cap\ldots\cap E_m) = d-m(d-k)\).
For this process, the total \((d-m(d-k))\)-dimensional volume \(F^{(m)}_r\) of the process in a ball \(B_r\) of radius \(r>0\) is studied.
We also investigate the asymptotic covariance matrix of the random vector \(F^{[m]}_r \coloneqq (F^{(1)}_r,\ldots,F^{(m)}_r)^\top\).

In euclidean space, after standardization, the analogous functionals \(F^{(m)}_{r,\textbf{e}}\) all converge in law to a unit normal distribution (see \cite{heinrich,LPST14}, \cite[Chap.~9]{HugSch24}, and the literature cited there). In hyperbolic space a completely different picture gradually emerged:
The problem was first treated for hyperplanes (i.e., for \(k=d-1\)) in \cite{HHT}.
There it is shown that \(F^{(m)}_{r}\) satisfies a standard CLT for \(d=2,3\).
In this case, rates of convergence are  established which turn out to depend on the dimension $d$ and on the number of intersections $m$.   
In addition, it is proved that \(F^{(m)}_{r}\) cannot be asymptotically normally distributed when \(d\geq 4\) and \(m=1\) or for \(d\geq 7\) and general \(m\).
Moreover it is shown that the asymptotic covariance matrix of the random vector \(F^{[d]}_{r}\) has rank two for \(d=2\) and rank one for \(d \geq 3\).
The paper \cite{HHT} develops various auxiliary results that proved to be crucial for   subsequent investigations, including the present work.

In \cite{KRT}, the asymptotic distribution of \(F^{(1)}_r\) was finally determined explicitly for hyperplanes (that is, for \(k = d-1\)). In this case, the limit law is infinitely divisible and has no Gaussian component. 
The paper \cite{BHT} investigates the behavior of intersection processes for general \(k\)-flats.
 Among other results, it is shown that a standard CLT holds when \(2k \leq d+1\), and convergence rates are provided in these cases.
It is noteworthy that the convergence rate becomes slower, the  closer \(2k\) is to \(d+1\)   (cf.\ \cite[Thm.~1.2]{BHT}).
Based on the analysis of the case \(k=d-1\) in \cite{KRT}, the asymptotic distribution of \(F^{(1)}_r\) was determined for \(2k > d+1\).
It is again infinitely divisible without Gaussian component.
For \(m \geq 2\), the limit distribution has so far remained unknown and in fact, it was not even known whether such a limit distribution exists at all, although, as pointed out above, it has been shown for certain \(d\) and \(k\) that it cannot be a normal distribution.

\medskip

In the present work, we determine the limit distribution for all remaining values of \(d,k,m\).
To state the result, take \(2k > d+1\) and let $\zeta$ be an inhomogeneous Poisson process on $ [0, \infty) $ with intensity function
given by $s \mapsto \omega_{d-k}\cosh^{k} s\,\sinh^{d-k-1} s$. Here, for \(j \in \N\),  \(\omega_j \coloneqq 2 \pi^{j/2} / \Gamma(j/2)\) denotes surface measure of the euclidean unit sphere of dimension \(j-1\). 
Let $Z$ be the infinitely divisible, centred random variable given by
\begin{equation}\label{eq:LimitVariableZ}
Z := \lim_{T\to\infty}\Big(\sum_{s\in\zeta\cap[0,T]}\cosh^{-(k-1)} s-\frac{\omega_{d-k}}{d-k}\sinh^{d-k} T\Big),
\end{equation}
where the convergence is both in \(L^2\) and a.s.\ (compare \cite[Rem.~1.5]{BHT}).
We show in \Cref{prop:limit_distribution} that
\begin{equation*}
    \frac{F^{(m)}_r - \E[F^{(m)}_r]}{e^{r(k-1)}} \xlongrightarrow{D}
        C(d,k,m) \frac{\omega_k}{(k-1)2^{k-1}} Z,
\end{equation*}
where the constant \(C(d,k,m)\) is given by \eqref{eq:def_C} below.
This confirms a conjecture from \cite{BHT} stating that \(F^{(m)}_r\) does not satisfy a standard CLT when \(2k > d+1\).
We also give explicit rates of convergence in the Kolmogorov distance.
\Cref{thm:rates_intersection} is the general result and \Cref{cor:rates} a simpler version, while Theorem \ref{thm:convergence_rates_1} treats the special case $m=1$.

We further show in \Cref{prop:covariance_matrix} that the covariance matrix of an appropriately scaled \(F^{[m]}_{r}\) has full rank when \(2k < d+1\) and rank one otherwise.
In the former case, we determine the covariance matrix, which is given by \eqref{eq:covariance_matrix} below.

\medskip

All results in this work are ultimately based  on a careful analysis of the Wiener--Itô chaos decomposition of \(F^{(m)}_r\).
The crucial observation is the identity \eqref{eq:decomp_F_m_general} that allows us to write
\[ F^{(m)}_r - \E[F^{(m)}_r] = C(d,k,m) (F^{(1)}_r - \E[F^{(1)}_r]) + W_r, \]
where the term \(W_r\) is asymptotically negligible as \(r \to \infty\) when \(2k > d+1\).
This reduces the problem of finding the asymptotic distribution of \(F^{(m)}_r\)  to the case \(m=1\), which has already been resolved in \cite{KRT} for $k=d-1$ and in \cite{BHT} for general $k$.

From the above identity it also follows that for \(m \geq 2\) the asymptotic covariance matrix of \(F^{[m]}\) has rank one if and only if \(W_r\) is of lower order than \(F^{1}_r\), which is the case precisely when \(2k \geq d+1\).

For the convergence rates, we use an inequality of Berry--Esseen type to relate closeness of the characteristic functions to closeness in Kolmogorov distance.
The inequalities needed for the characteristic functions are obtained by means of basic calculus.

\medskip

The paper is structured as follows. 
In \Cref{sec:preliminaries}, the main objects of interest are defined and some notation is introduced.
We also collect auxiliary results that will be needed later on.

In \Cref{sec:comparison}, we compare the euclidean and the hyperbolic setting and explain some of the underlying reasons for the different behavior observed in non-euclidean space.

The following three sections deal with the case \(2k > d+1\), where \(F^{(m)}_r\) is not asymptotically normally distributed.
In \Cref{sec:qualitative}, the limit distribution is determined, then convergence rates are established in \Cref{sec:convergence_rates}.
In \Cref{sec:density_approximation}, the density of the limit distribution is approximated numerically and the asymptotic dependence of the standardized cumulants of the limit distribution $Z_{d,k}$ on the dimension $d$ of the space and the dimension $k$ of the flats is analyzed.

Finally, we investigate the behavior of the asymptotic covariance matrices in \Cref{sec:covariance_matrix}.

\section{Preliminaries}
\label{sec:preliminaries}

In this section, we recall some notation and provide background information from \cite{BHT} relevant for the present purpose.
For further details we refer to \cite{HHT} and the literature cited there. 

Let \(d \geq 2\) and \(k\in\{0,\ldots,d-1\}\).
We  denote the space of \(k\)-flats in \(\Hy^d\) and \(\R^d\) by \(A_\textbf{h}(d,k)\) and \(A_\textbf{e}(d,k)\) respectively.
For more information on these spaces we refer the reader to \cite{BHT}.
General information on hyperbolic geometry can be found in \cite{Ra19}.
Since our main focus is on hyperbolic space, we introduce the convention \(A(d,k) \coloneqq A_\textbf{h}(d,k)\).
From here on, we will use this convention for any object we define in both hyperbolic and euclidean space and only explicitly use the index `\(\textbf{h}\)' in order to emphasize that we are working in \(\Hy^d\).

For \(\bullet \in \{\textbf{h},\textbf{e}\}\), let \(\mu_k\) denote the isometry invariant measure on \(A_\bullet(d,k)\), normalized as in \cite{BHT}.
We assume that the intensity parameter \(t\) used in \cite{BHT} is  equal to \(1\).
A Poisson process on \(A_\bullet(d,k)\) with intensity measure $\mu_k$  will be denoted by \(\eta\).
For \(m \in \N\) with \(d-m(d-k)\ge 0\) and \(r>0\) we study the functional
\[ F^{(m)}_{r,\bullet} = \frac{1}{m!}\sum_{(E_1,\ldots,E_m) \in \eta^m_{\neq}} \mathcal{H}_\bullet ^{d-m(d-k)}(E_1\cap\ldots\cap E_m \cap B_r)
    \one\{ \dim(E_1\cap\ldots\cap E_m) = d - m(d-k) \}, \]
where \(\mathcal{H}^i_\bullet\) is the \(i\)-dimensional Hausdorff measure, \(B_r\) is a ball of radius \(r\) with some arbitrary center in \(\Hy^d\)/\(\R^d\) and the sum extends over all \(m\)-tuples of distinct $k$-flats \(E_1,\ldots,E_m \in \eta\).
Hence \(F^{(m)}_{r,\bullet}\) is the total $(d-m(d-k))$-volume of the \(m\)-th order intersection process of \(\eta\) in a ball of radius $r$.
Clearly, for each $r>0$  the functional $F^{(m)}_{r,\bullet}$ is a Poisson $U$-statistic of order $m$.

With the definition
\[ f^{(m)}_{r,\bullet}(E_1,\ldots,E_m) \coloneqq \mathcal{H}_\bullet^{d-m(d-k)}(E_1\cap\ldots\cap E_m \cap B_r)
    \one\{ \dim(E_1\cap\ldots\cap E_m) = d - m(d-k) \},  \]
for \(E_1,\ldots,E_m \in A_\bullet(d,k)\), we have
\[ F^{(m)}_{r,\bullet} = \frac{1}{m!}\sum_{(E_1,\ldots,E_m) \in \eta^m_{\neq}} f^{(m)}_{r,\bullet}(E_1,\ldots,E_m).\]
As in \cite[Eq.~(2.3)]{BHT} we get the Wiener--Itô chaos decomposition
\begin{equation}
\label{eq:chaos_decomp}
    F^{(m)}_{r,\bullet} = \E[F^{(m)}_{r,\bullet}] + I_1(f^{(m)}_{r,1,\bullet}) + \cdots + I_m(f^{(m)}_{r,m,\bullet}),
\end{equation}
where the functions \(f^{(m)}_{r,i,\bullet}\), \(i=1,\ldots,m\), are defined by
\[\begin{aligned}
    &f^{(m)}_{r,i,\bullet} (E_1,\ldots,E_i) \coloneqq\\
    &\qquad\binom{m}{i} \frac{1}{m!} \int_{A(d,k)^{m-i}} f^{(m)}_{r,\bullet} (E_1, \ldots, E_i, E_{i+1}, \ldots, E_m) \,\mu_k^{m-i}(\dint\,(E_{i+1},\ldots,E_{m})),
\end{aligned} \]
for \(E_1,\ldots, E_i \in A_\bullet(d,k)\).
By the isometry property of the Itô integral, it follows that the summands on the right-hand side of  \eqref{eq:chaos_decomp} are uncorrelated,  \(\E[I_i(f^{(m)}_{r,i})] = 0\) and
\begin{equation*}
    \V(I_i(f^{(m)}_{r,i,\bullet})) = i! \int_{A_\bullet(d,k)^i} (f^{(m)}_{r,i,\bullet}(E_1,\ldots,E_i))^2 \,\mu^i_k(\dint\,(E_1,\ldots,E_i)) \eqqcolon i! \cdot A^{(m)}_{r,i,\bullet},
\end{equation*}
for \(i=1,\ldots,m\).
This implies
\[ \V(F^{(m)}_{r,\bullet}) = \sum_{i=1}^m \V(I_i(f^{(m)}_{r,i,\bullet})) = \sum_{i=1}^m i! A^{(m)}_{r,i,\bullet}, \]
which is precisely the statement of \cite[Eq.~(2.8)]{BHT}.

By \cite[Eq.~(2.10)]{BHT} it holds that
\begin{equation}
\label{eq:crofton_applied}
    \begin{aligned}
        f^{(m)}_{r,1,\bullet}(E)
            &= \binom{m}{1} \frac{1}{m!} \left(\frac{\omega_{d+1}}{\omega_{k+1}}\right)^{m-1} \frac{\omega_{d-m(d-k)+1}}{\omega_{d-(d-k)+1}}
                \mathcal{H}_\bullet^{d-(d-k)}(E \cap B_r)\\
            &= C(d,k,m) \mathcal{H}_\bullet^{k}(E \cap B_r),
    \end{aligned}
\end{equation}
for  \(E \in A_\bullet(d,k)\), where
\begin{equation}
    \label{eq:def_C}
    C(d,k,m) \coloneqq \frac{1}{(m-1)!}  \left(\frac{\omega_{d+1}}{\omega_{k+1}}\right)^{m-1} \frac{\omega_{d-m(d-k)+1}}{\omega_{k+1}}.
\end{equation}
By the definition of \( F^{(1)}_{r,\bullet}\) we have
\begin{equation}
\label{eq:identity_first_term}
    I_1(f^{(m)}_{r,1\bullet})
        = C(d,k,m) \left( F^{(1)}_{r,\bullet} - \E[F^{(1)}_{r,\bullet}]  \right).
\end{equation}

The following statement appears as Lemma 2.5 in \cite{BHT}. We include it here for easier reference and comparison.
The notation \(A_r \asymp B_r\) is to be understood as \(c \cdot B_r \leq A_r \leq C \cdot B_r\)
for all \(r\geq 1\) and some positive constants \(c, C\).
Similarly, we adopt the notation \(A_r \lesssim B_r\) (\(A_r \gtrsim B_r\)) when \(A_r \leq C \cdot B_r\) (\(A_r \geq C \cdot B_r\))
for all \(r \geq 1\) and some positive constant \(C\).
Unless explicitly stated otherwise, the constants \(c,C\) may depend on \(d,k\).
\begin{lemma}
\label{lem:A_asymptotics}
   If \(r\geq 1\), \(k\in\{0,\ldots,d-1\}\) and \(i\in\{1,\ldots,m\}\), then
    \begin{align*}
        A_{r,i,\textbf{h}}^{(m)} \ &\asymp
        \begin{cases}
            e^{r(d-1)}&: \, 2i(d-k)> d-1,\\
            re^{r(d-1)}&: \, 2i(d-k)= d-1,\\
            e^{2r(d-i(d-k)-1)}&: \, 2i(d-k)<d-1.
        \end{cases}
    \end{align*}
\end{lemma}

\section{Comparison with the euclidean case}
\label{sec:comparison}

The euclidean analogue of \Cref{lem:A_asymptotics} is
\begin{equation}
\label{eq:A_asymptotics_euc}
    A^{(m)}_{r,i,\textbf{e}} \asymp r^{2d - i(d-k)}.
\end{equation}
To prove this, use \cite[Eq.~(2.9)]{BHT} together with the fact that
\[ \int_{A_\textbf{e}(d,k)} \mathcal{H}^k_\textbf{e}(E \cap B_r)^l \mu_k(\dint E) \asymp r^{d+k(l-1)} \]
for \(0\leq k<d\) and \(l\in\N\).
This is the euclidean analogue of \cite[Lem.~8]{HHT}.
Since the proof is conceptually identical and technically easier in euclidean space, we omit it here.

Using the chaos decomposition \eqref{eq:chaos_decomp} and \eqref{eq:identity_first_term}, we obtain
\begin{equation}
\label{eq:decomp_F_m_general}
    F^{(m)}_{r,\bullet} - \E[F^{(m)}_{r,\bullet}]
        = C(d,k,m) (F^{(1)}_{r,\bullet} - \E[F^{(1)}_{r,\bullet}]) + \sum_{i=2}^m I_i(f^{(m)}_{r,i,\bullet})
\end{equation}
for \(\bullet\in\{\textbf{e}, \textbf{h}\}\).
For \(\bullet = \textbf{e}\), it follows from \eqref{eq:A_asymptotics_euc} that the first term in this decomposition has variance of order \(r^{d+k}\),
while the remaining terms have variance of order at most \(r^{d+k-(d-k)}\).
Hence, \(\V(F^{(m)}_{r,\textbf{e}}) \asymp r^{d+k}\) and we can write
\[ \frac{F^{(m)}_{r,\textbf{e}} - \E[F^{(m)}_{r,\textbf{e}}]}{\sqrt{r^{d+k}}} = C(d,k,m) \frac{F^{(1)}_{r,\textbf{e}} - \E[F^{(1)}_{r,\textbf{e}}]}{\sqrt{r^{d+k}}} + W_r, \]
with \(W_r \xlongrightarrow{\P} 0\) as \(r \to \infty\) (note that \(\V(W_r) \lesssim r^{-(d-k)}\)).
From these facts we can draw two conclusions:
Firstly, if \(X\) is a random variable with
\[ \frac{F^{(1)}_{r,\textbf{e}} - \E[F^{(1)}_{r,\textbf{e}}]}{\sqrt{r^{d+k}}} \xlongrightarrow{D} X, \]
then (by Slutzky's lemma)
\[ \frac{F^{(m)}_{r,\textbf{e}} - \E[F^{(m)}_{r,\textbf{e}}]}{\sqrt{r^{d+k}}} \xlongrightarrow{D} C(d,k,m) X. \]
Secondly, the asymptotic covariance matrix
\[\begin{aligned}
     \lim_{r\to\infty} \frac{1}{r^{d+k}} \Sigma(F^{(1)}_{r,\textbf{e}}, \ldots, F^{(m)}_{r,\textbf{e}})
    &= \lim_{r\to\infty} \frac{\V(F^{(1)}_{r,\textbf{e}})}{r^{d+k}} \cdot \Big( C(d,k,a)C(d,k,b) \Big)_{a,b=1,\ldots,m}
\end{aligned}\]
has rank \(1\).
For \(k=d-1\) this was shown in \cite[Thm.~5.1(i)]{heinrich}, although the above reasoning implies that this is true for any \(k<d\).
It is well known that \(X\) follows a normal distribution, however this is not the point we want to make here.

The behavior described above is in contrast to the hyperbolic setting:
There we have \(W_r \xlongrightarrow{\P} 0\) only in the case \(2k \geq d+1\) (compare \Cref{lem:A_asymptotics}).
For \(2k<d+1\), the terms \(A^{(m)}_{r, 1, \textbf{h}},\ldots,A^{(m)}_{r,m,\textbf{h}}\) are all of the same order \(e^{r(d-1)}\),
hence all terms of the chaos decomposition contribute to the asymptotic behavior.

This is reflected in the fact that for \(k=d-1\), the hyperbolic analogue of the asymptotic covariance matrix has full rank for \(2k < d+1\) (i.e., \(d=2\))
and rank \(1\) for \(2k \geq d+1\) (i.e., \(d\geq 3\)) which is shown in \cite[Sec.~4.5]{HHT}.
We extend this result to \(k<d\) in \Cref{sec:covariance_matrix}.

The fact that the first term of the chaos expansion determines the asymptotic behavior for \(2k \geq d+1\) allows us to reduce
the problem of finding the asymptotic distribution of \(F^{(m)}_{r,\textbf{h}}\) to the case \(m=1\).
From \eqref{eq:decomp_F_m_general} it follows almost immediately that the limit distribution of \(e^{-r(k-1)}(F^{(m)}_{r,\textbf{h}} - \E[F^{(m)}_{r,\textbf{h}}])\) differs from the one of \(e^{-r(k-1)}(F^{(1)}_{r,\textbf{h}} - \E[F^{(1)}_{r,\textbf{h}}])\) only by a constant factor.
The latter distribution has been determined for \(k = d-1\) in \cite{KRT} and for general \(k<d\) in \cite{BHT}.
Curiously, it is not a normal distribution.

In summary, one can say that in the hyperbolic setting there are two phenomena which lead to behavior that diverges from what we know in the euclidean setting:
Firstly, for \(2k < d+1\), all terms of the chaos distribution have variance of the same order, leading to a non-degenerate covariance matrix.
Secondly, for \(2k > d+1\), \(F^{(1)}_{r,\textbf{h}}\) follows a nonstandard central limit theorem, leading to nonstandard limits for the functionals of intersection processes \(F^{(m)}_{r,\textbf{h}}\), \(m \geq 2\).
Note that the second phenomenon has nothing to do with intersection processes per se, since for \(m=1\) no intersections are considered.

\section{Determination of the limit distribution}
\label{sec:qualitative}

We now determine the asymptotic distribution of \(F^{m}_r\), thus resolving in particular the conjecture stated in \cite{BHT}.
Note that the following proposition is a purely qualitative statement, convergence rates will be derived in the next section.

\begin{proposition}
\label{prop:limit_distribution}
If $2k>d+1$ and \(d-m(d-k) \geq 0\), then
\begin{equation*}
    \frac{F^{(m)}_r - \E[F^{(m)}_r]}{e^{r(k-1)}} \xlongrightarrow{D}
        C(d,k,m) \frac{\omega_k}{(k-1)2^{k-1}} Z
\end{equation*}
as $r\to\infty$, where \(Z\) is given by \eqref{eq:LimitVariableZ} and \(C(d,k,m)\) is given by \eqref{eq:def_C}.
\end{proposition}
\begin{proof}
Since \(2k > d+1\), it follows from \Cref{lem:A_asymptotics} that \(A^{(m)}_{r,1} \asymp e^{2r(k-1)}\) and that
\(A^{(m)}_{r,i}\) is of lower order for \(i=2,\ldots,m\), i.e., \( A^{(m)}_{r,i}/e^{2r(k-1)} \to 0, \) as \(r\to \infty\).
This implies that \( \V(I_i(f^{(m)}_{r,i}) / e^{r(k-1)} )\) goes to \(0\) for \(r \to \infty\) and hence
\begin{equation*}
    I_i(f^{(m)}_{r,i}) / e^{r(k-1)} \xlongrightarrow{\P} 0, \qquad i=2,\ldots,m.
\end{equation*}
Using \eqref{eq:decomp_F_m_general}, we can write
\begin{equation}
\label{eq:decomp_F_m}
    \begin{aligned}
        \frac{F^{(m)}_r - \E[F^{(m)}_r]}{e^{r(k-1)}}
            &= C(d,k,m) \frac{F^{(1)}_r - \E[F^{(1)}_r]}{e^{r(k-1)}} + \sum_{i=2}^m \frac{I_i(f^{(m)}_{r,i})}{e^{r(k-1)}}.
    \end{aligned}
\end{equation}
The first term converges in distribution to \(C(d,k,m) \frac{\omega_k}{(k-1)2^{k-1}} Z\) by \cite[Thm.~1.4]{BHT} (note that the prefactor contains a typo which we corrected here)
and the second term converges to \(0\) in probability by the previous consideration.
The result now follows by Slutsky's lemma.
\end{proof}

\section{Convergence rates}
\label{sec:convergence_rates}

Having established the limit distribution, we now work a bit harder to get bounds on the speed of convergence.
Throughout this whole section, let \(2k > d+1\) and \(d-m(d-k) \geq 0\).

We start by treating the case \(m=1\).
For this purpose, let
\[ Y_r \coloneqq \frac{F^{(1)}_r - \E[F^{(1)}_r]}{e^{r(k-1)}}, \qquad Y \coloneqq \frac{\omega_k}{(k-1)2^{k-1}} Z, \]
for \(r > 0\).

\begin{theorem}
\label{thm:convergence_rates_1}
    If
    \[ \beta_{d,k} \coloneqq \frac{2(2k -d -1)}{k+d+4},\]
    then there is a constant \(C>0\) depending only on \(d,k\) such that
    \[ d_K(Y_r, Y) \leq C \cdot \exp\left(- \beta_{d,k}\cdot r\right), \]
    for \(r>0\) when \(k \geq 4\), and
    \[ d_K(Y_r, Y) \leq C \cdot (r+1) \cdot \exp\left(-\beta_{d,k}\cdot r\right), \]
    for \(r>0\) and \(k=3\).
\end{theorem}

\begin{remark}
    Note that the restrictions on \(k\) and \(d\) ensure that \(k\) is always at least \(3\) and that \(k=3\) is only possible when \(d=4\).
    Note further that \(0<\beta_{d,k}<1\) for all admissible values of \(d,k\).
    For fixed \(d\), \(\beta_{d,k}\) increases for increasing admissible \(k\).
    Roughly speaking, it holds for large values of \(d\) that \(\beta_{d,k} \approx 0\) when \(k \approx d/2\) and \(\beta_{d,k} \approx 1\) when \(k \approx d\).
\end{remark}

In the proof of Theorem \ref{thm:convergence_rates_1} we will make use of the fact that for \(r > 0\) the characteristic functions of \(Y\) and \(Y_r\) have been determined in \cite{BHT} as
\[\begin{aligned}
    \varphi_r(t) \coloneqq \varphi_{Y_r}(t) &= \exp\left( \omega_{d-k} \int_0^\infty (e^{it g_r(s)} - 1 - it g_r(s))\, \mu(\dint s) \right),\\
    \varphi(t) \coloneqq \varphi_{Y}(t) &= \exp\left( \omega_{d-k} \int_0^\infty (e^{it g(s)} - 1 - it g(s))\, \mu(\dint s) \right),
\end{aligned}\]
for \(t\in\R\), where \(g, g_r \colon [0,\infty) \to [0,\infty)\), for \(r\geq 0\), are given by
\[\begin{aligned}
    g_r(s) &= \one\{s\leq r\} \omega_k e^{-(k-1)r} \int_0^{\arcosh\left(\frac{\cosh(r)}{\cosh(s)}\right)} \sinh^{k-1}(u) \,\dint u,\\
    g(s) &= \frac{\omega_k}{(k-1)2^{k-1}} \cosh^{-(k-1)}(s),
\end{aligned}\]
for \(s \geq 0\), and we use the abbreviation \(\mu(\dint s) = \cosh^k(s) \sinh^{d-k-1}(s) \, \dint s\).
For a fixed \(s\geq 0\) it has already been established in \cite[Lem.~4.1]{BHT} that \(g_r(s) \to g(s)\) as \(r \to \infty\).
Later we will quantify the convergence.

To transfer closeness of the characteristic functions to closeness of the distribution functions,
we use the following inequality of Berry--Esseen type.

\begin{theorem}
    \label{thm:berry_esseen}
        Let \(F, G\) be distribution functions with  characteristic functions \(\varphi_F, \varphi_G\) respectively.
        If \(G\) has bounded density \(g \leq M\), then
        \[ \abs{F(x) - G(x)} \leq \frac{1}{\pi} \int_{-T}^T \frac{\abs{\varphi_F(t) - \varphi_G(t)}}{\abs{t}} \, \dint t + \frac{24 M}{\pi T} \]
        for  \(x \in \R\) and \(T > 0\).
\end{theorem}
This inequality is originally due to Esseen, it appears as Theorem II.2.a in \cite{Esseen1945FourierAO}, albeit in a slightly different form .
The version presented above can be found in \cite[p.~285]{loeve}.
A less general version of the result is also contained in \cite[Sec.~3.4.4]{Durrett}.

\subsection{Bounds on \texorpdfstring{\(g\) and \(g_r\)}{g and g\_r}}

To prepare the proof of \Cref{thm:convergence_rates_1}, we provide further auxiliary results. We start by rewriting the function $g_r$.

\begin{lemma}
    If \(r,s \geq 0\), then
    \begin{equation}
        \label{eq:g_identity_1}
        g_r(s) = \one\{s\leq r\}\omega_k e^{-(k-1)r} \cosh^{-(k-1)}(s) \int_s^r (\sinh^2(u) - \sinh^2(s))^\frac{k-2}{2} \sinh(u)\, \dint u.
    \end{equation}
\end{lemma}
\begin{proof}
    Note that the equality is clear for \(r \leq s\).
    On \(\{(r,s) \colon 0 \le s < r\}\) application of the chain rule yields
    \[\begin{aligned}
        &\frac{\partial}{\partial r} \int_0^{\arcosh\left(\frac{\cosh(r)}{\cosh(s)}\right)} \sinh^{k-1}(u) \,\dint u\\
            &\qquad= \sinh\left(\arcosh\left(\frac{\cosh(r)}{\cosh(s)}\right)\right)^{k-1} \frac{\partial}{\partial r} \arcosh\left(\frac{\cosh(r)}{\cosh(s)}\right)\\
            &\qquad= \left(\frac{\cosh^2(r)}{\cosh^2(s)} - 1\right)^\frac{k-1}{2} \left(\frac{\cosh^2(r)}{\cosh^2(s)} - 1\right)^{-\frac{1}{2}} \frac{\sinh(r)}{\cosh(s)}\\
            &\qquad= \frac{(\cosh^2(r) - \cosh^2(s))^\frac{k-2}{2}}{\cosh^{k-1}(s)} \sinh(r)\\
            &\qquad= \frac{(\sinh^2(r) - \sinh^2(s))^\frac{k-2}{2}}{\cosh^{k-1}(s)} \sinh(r),
    \end{aligned}\]
    where we have used the identity \(\cosh^2 - \sinh^2 = 1\) repreatedly.
    Hence
    \[\int_0^{\arcosh\left(\frac{\cosh(r)}{\cosh(s)}\right)} \sinh^{k-1}(u)\, \dint u - 0
        =\int_s^r  \frac{(\sinh^2(t) - \sinh^2(s))^\frac{k-2}{2}}{\cosh^{k-1}(s)} \sinh(t) \, \dint t,\]
    which implies the assertion of the lemma.
\end{proof}

Next we show that $g_r(s)$ approaches $g(s)$  from below, for fixed $s\ge 0$, and derive a bound for the speed of convergence as $r\to\infty$.

\begin{lemma}
    If \(r, s \geq 0\), then
    \begin{equation}
    \label{eq:g_bound_1}
        g_r(s) \leq g(s)
    \end{equation}
    and
    \begin{equation}
    \label{eq:g_bound_2}
        g(s) - g_r(s) \leq \frac{\omega_k}{k-1} e^{-(k-1)r} +\one\{s\leq r\} \omega_k (k-1) e^{-(k-1)r-(k-3)s} \int_s^r e^{(k-3)u} \,\dint u.
    \end{equation}
\end{lemma}
\begin{proof}
    Since the first inequality is clear for \(s > r\), assume that \(s \leq r\).
    From the identity \eqref{eq:g_identity_1} and the bound
    \[ (\sinh^2(u) - \sinh^2(s))^\frac{k-2}{2} \sinh(u) \leq \sinh(u)^{k-1} \leq 2^{-(k-1)}e^{(k-1)u} \]
    for \(0 \leq s \leq u\), we see that
    \[\begin{aligned}
        g_r(s)
            &\leq \omega_k e^{-(k-1)r} \cosh^{-(k-1)}(s) \int_s^r 2^{-(k-1)} e^{(k-1) u} \,\dint u\\
            &=\frac{\omega_k}{2^{k-1}} e^{-(k-1)r} \cosh^{-(k-1)}(s) \frac{1}{k-1} (e^{(k-1)r} - e^{(k-1)s})\\
            &\leq\frac{\omega_k}{2^{k-1}(k-1)} \cosh^{-(k-1)}(s)\\
            &=g(s),
    \end{aligned}\]
    which proves \eqref{eq:g_bound_1}.

    For \(s>r\), the second inequality holds since
    \[ g(s) - g_r(s) = g(s) \leq g(r) = \frac{\omega_k}{k-1} (2\cosh(r))^{-(k-1)} \leq \frac{\omega_k}{k-1} e^{-(k-1)r}, \]
    where we used that \(2\cosh(r)\geq e^r\).
    Assume now that \(s\leq r\).
    Using relation \eqref{eq:g_identity_1}, $g(s)-g_r(s)\ge 0$ (see \eqref{eq:g_bound_1}) and \(2\cosh(s)\geq e^s\), we obtain
    \begin{align}
        &g(s) - g_r(s)\nonumber\\
        &= \frac{\omega_k}{(2\cosh(s))^{k-1}} e^{-(k-1)r}  
                \left( \frac{e^{(k-1)r}}{k-1} - \int_s^r 2^{k-1}(\sinh^2(u) - \sinh^2(s))^\frac{k-2}{2} \sinh(u)\,\dint u \right)\nonumber\\
        &\le  \omega_k e^{-(k-1)(r+s)} 
                \left( \frac{e^{(k-1)s}}{k-1} +  \int_s^r e^{(k-1)u} - 2^{k-1}(\sinh^2(u) - \sinh^2(s))^\frac{k-2}{2} \sinh(u)\,\dint u \right)\nonumber\\
        &\le  \omega_k e^{-(k-1)(r+s)} 
                \left( \frac{e^{(k-1)s}}{k-1} +  \int_s^r (e^{2u})^{\frac{k-1}{2}} - (4 \sinh^2(u) - 4 \sinh^2(s))^\frac{k-1}{2} \,\dint u \right),     \label{eq:integrand}   
    \end{align}
    Recall that $k\ge 3$.  The mean value theorem yields
    \[ y^\frac{k-1}{2} - x^\frac{k-1}{2} \leq (y - x) \frac{k-1}{2} y^\frac{k-3}{2} \]
    for \(0\leq x \leq y\).
    Applying this with \(y = e^{2u}\) and \(x = 4 \sinh^2(u) - 4 \sinh^2(s)\) then allows us to bound the integrand in \eqref{eq:integrand} from above by
    \[\begin{aligned}
        (e^{2u} - (4 \sinh^2(u) - 4 \sinh^2(s))) \frac{k-1}{2} (e^{2u})^\frac{k-3}{2}
            &= (e^{2s} + e^{-2s} - e^{-2u}) \frac{k-1}{2} e^{(k-3)u}\\
            &\leq (k-1) e^{2s} e^{(k-3)u}.
    \end{aligned}\]
    Plugging this bound into \eqref{eq:integrand}, we arrive at
    \[\begin{aligned}
        g(s) - g_r(s)
                &\leq  \omega_k  e^{-(k-1)(r+s)}   \left( \frac{e^{(k-1)s}}{k-1} +  \int_s^r (k-1) e^{2s} e^{(k-3)u} \, \dint u \right)\\
                &= \frac{\omega_k}{k-1} e^{-(k-1)r} + \omega_k (k-1) e^{-(k-1)r-(k-3)s} \int_s^r e^{(k-3)u} \,\dint u,
    \end{aligned}\]
    which concludes the proof.
\end{proof}

\begin{lemma}
    \label{lem:bound_int}
        There exists some constant \(C_{d,k} > 0\) such that, for any
        \(s_0,r \geq 0\),
        \begin{equation*}
            \int_0^{s_0} \abs{g(s) - g_r(s)} \, \mu(\dint s) \leq C_{d,k} ( e^{-(k-1)r + (d-1)s_0} + e^{-2r + (d-k+2)s_0} )
        \end{equation*}
        when \(k \geq 4\) and
        \begin{equation*}
            \int_0^{s_0} \abs{g(s) - g_r(s)} \, \mu(\dint s) \leq C_{d,k} ( e^{-(k-1)r + (d-1)s_0} + r e^{-2r + (d-k+2)s_0} )
        \end{equation*}
        in the case \(k=3\) (and thus \(d=4\)), as well as
        \begin{equation*}
            \int_{s_0}^\infty \abs{g_r(s)}^2 + \abs{g(s)}^2 \, \mu(\dint s) \leq C_{d,k} e^{-(2k-d-1)s_0},
        \end{equation*}
        for any \(k, d\).
    \end{lemma}
    \begin{proof}
       For \(r \geq 0\) we define \(h_3(r): = r\) and \(h_k(r) := 1\) when \(k \geq 4\). Then
        \[ \int_s^r e^{(k-3)u}\, \dint u\leq h_k(r) e^{(k-3)r} \]
        for all \(k\geq 3\) and \(s,r\geq 0\).
        Combining this with the fact that \(g(s) - g_r(s)\ge 0\)  (see \eqref{eq:g_bound_1}), relation \eqref{eq:g_bound_2}
        and the bound \(\mu(\dint s) \leq e^{(d-1)s}\,\dint s\), we get
        \[\begin{aligned}
            &\int_0^{s_0} \abs{g(s) - g_r(s)} \,\mu(\dint s)\\
                &\qquad\leq \int_0^{s_0} \frac{\omega_k}{k-1} e^{-(k-1)r} + \omega_k (k-1) e^{-2r-(k-3)s} h_k(r)\, \mu(\dint s)\\
                &\qquad= \frac{\omega_k}{k-1} e^{-(k-1)r} \int_0^{s_0} \, \mu(\dint s) + \omega_k (k-1) e^{-2r} h_k(r) \int_0^{s_0} e^{-(k-3)s} \,  \mu(\dint s)\\
                &\qquad\leq \frac{\omega_k}{k-1} e^{-(k-1)r} \frac{1}{d-1} e^{(d-1)s_0} + \omega_k (k-1) e^{-2r} h_k(r) \frac{1}{d-k+2} e^{(d-k+2)s_0},
        \end{aligned}\]
        from which the first two claims follow immediately.

        For the last bound, simply observe that  \eqref{eq:g_bound_1} implies \(\abs{g_r(s)}^2 \leq \abs{g(s)}^2\) and thus
        \[\begin{aligned}
            \int_{s_0}^\infty \abs{g_r(s)}^2 + \abs{g(s)}^2 \,\mu(\dint s)
                &\leq 2 \int_{s_0}^\infty \left(\frac{\omega_k}{k-1} (2\cosh(s))^{-(k-1)}\right)^2\, \mu(\dint s)\\
                &\leq2\int_{s_0}^\infty \left(\frac{\omega_k}{k-1} e^{-(k-1)s}\right)^2 e^{(d-1)s} \, \dint s\\
                &= \frac{2\omega_k^2}{(k-1)^2(2k-d-1)} e^{-(2k-d-1)s_0},
        \end{aligned}\]
        where we used that \(2 \cosh(r) \geq e^r\) and \(\mu(\dint s) \leq e^{(d-1)s}\dint s\) in the second line.
    \end{proof}

\subsection{Proof of Theorem \ref{thm:convergence_rates_1}}

In order to apply \Cref{thm:berry_esseen} to the random variables \(Y_r\) and \(Y\), we first need to establish that \(Y\) has bounded density.

\begin{lemma}
\label{lem:bounded_density}
    The distribution of \(Y\) has bounded density.
\end{lemma}
\begin{proof}
    Note that \(Y\) is a multiple of \(Z\), so it suffices to show that \(Z\) has a bounded density.
    The argument we use appears in \cite{orey}, we include it here since it does not take up a lot of lines and might offer the reader some insight.
    In \cite[Rem.~1.5]{BHT}, it is established that \(Z\) is infinitely divisible without Gaussian component and further that its Lévy measure \(\nu\) is concentrated on \((0,1)\) with Lebesgue density
    \[ \rho(y) = \frac{\omega_{d-k}}{k-1} y^{-\frac{d+k-2}{k-1}} \left(1-y^\frac{2}{k-1}\right)^{\frac{d-k}{2}-1},\qquad y \in (0,1).  \]
    Hence there is a constant \(c>0\) such that
    \begin{equation}\label{lbound}
         \int_0^{s} y^2 \nu(\dint y) \geq c\, s^{\frac{2k-d-1}{k-1}} \geq c\, s 
    \end{equation}
    for all \(s \in (0,1/2)\).

    By the Lévy--Khintchine formula, the characteristic function of \(Z\) satisfies
    \[ -\log|\varphi_Z(t)| = \int_0^1 (1-\cos(ty))\, \nu(\dint y),\qquad t \in \R. \]
    Using \eqref{lbound} and the inequality \(1-\cos(x) \geq x^2/4\), \(x \in [0,1]\), we arrive at
    \begin{align*}
        -\log|\varphi_Z(t)| \geq \int_0^\frac{1}{|t|} (1-\cos(ty)) \, \nu(\dint y)
        \geq \frac{1}{4}t^2 \int_0^\frac{1}{|t|} y^2 \,\nu(\dint y) \geq \frac{c}{4} |t|,
    \end{align*}
    for \(|t| \geq 2\).
    Hence \( |\varphi_Z(t)| \leq \exp(-\frac{c}{4}|t|) \) holds for \(t\geq 2\) and by the inversion formula (see, e.g., \cite[Thm.~3.3.14]{Durrett}) \(Z\) admits a bounded probability density.
\end{proof}

We can now apply \Cref{thm:berry_esseen} and obtain that for any \(r > 0\) and \(T > 0\) it holds that
\begin{equation}
\label{eq:bound_kol_1}
    d_K(Y_r, Y) \leq \frac{1}{\pi} \int_{-T}^T \frac{\abs{\varphi_r(t) - \varphi(t)}}{\abs{t}}\, \dint t + \frac{24 M}{\pi T},
\end{equation}
where \(M > 0\) is a constant that depends only on \(d,k\).

To proceed, we derive bounds for the first term on the right-hand side of relation \eqref{eq:bound_kol_1}.
First, we define \(q\colon \R \to \C\), \(x \mapsto e^{i x} - 1 - ix\), which allows us to express the characteristic functions as
\[  \varphi_r(t) = \exp\left( \omega_{d-k} \int_0^\infty q(t g_r(s)) \,\mu(\dint s) \right),\qquad
    \varphi(t) = \exp\left( \omega_{d-k} \int_0^\infty q(t g(s)) \,\mu(\dint s) \right),\]
for \(r > 0\).
Note that \(\Real(q(x)) \leq 0\) for all \(x\in\R\).
For complex numbers \(z,w\) with real part \(\leq 0\), it holds that
\begin{equation}
\label{eq:bound_exp}
    \abs{e^z - e^w} \leq \abs{z-w}.
\end{equation}
(To see this, let \(\gamma \colon [0,1] \to \C, t\mapsto \exp(w+t(z-w))\).
Note that \(\Real(w+t(z-w)) \leq 0\) for all \(t\in[0,1]\), hence the fundamental theorem of calculus and the triangle inequality give
\(\abs{e^z-e^w} \leq \int_0^1 \abs{\gamma'(t)} \dint t = \int_0^1 \abs{(z-w)\exp(w+t(z-w))} \dint t \leq \abs{z-w} \).)
Applying this to the characteristic functions and using the triangle inequality, we obtain
\begin{equation}
\label{eq:bound_CF_1}
    \abs{\varphi_r(t) - \varphi(t)} \leq \omega_{d-k} \int_0^\infty \abs{q(t g_r(s)) - q(t g(s))}\, \mu(\dint s),
\end{equation}
for \(r > 0\).

\begin{lemma}
\label{lem:bound_q}
    If \(x,y\in\R\), then \(\abs{q(y) - q(x)} \leq 2 \abs{y - x}\) and \(\abs{q(x)} \leq \frac{1}{2} \abs{x}^2\).
\end{lemma}
\begin{proof}
    From \eqref{eq:bound_exp} we get
    \[ \abs{q(y) - q(x)} \leq \abs{ e^{iy} - e^{ix}} + \abs{iy - ix} \leq 2 \abs{y-x}, \]
    which proves the first assertion.
    The second inequality follows from \cite[Lem.~6.15]{kallenberg}.
\end{proof}
\noindent
For an arbitrary \(s_0 > 0\), combining \eqref{eq:bound_CF_1} and \Cref{lem:bound_q} yields
\begin{equation}
\label{eq:bound_CF_2}
\begin{aligned}
    \abs{\varphi_r(t) - \varphi(t)}
        &\leq 2\omega_{d-k} \abs{t} \int_0^{s_0}  \abs{g_r(s) - g(s)}\, \mu(\dint s) \\
        &\qquad + \frac{1}{2}\omega_{d-k} \abs{t}^2 \int_{s_0}^\infty \abs{g_r(s)}^2 + \abs{g(s)}^2 \,\mu(\dint s),
\end{aligned}
\end{equation}
for \(r>0\).

For the sake of readability, we will from now on simply write \(C\) for a universal positive constant which may depend on \(d,k\).
Constants may also change from line to line.

Using \eqref{eq:bound_CF_2} and \Cref{lem:bound_int}, the first term of \eqref{eq:bound_kol_1} can be bounded for any \(s_0 \geq 0\) and \(r>0\) by
\[\begin{aligned}
    &\frac{1}{\pi} \int_{-T}^T \frac{\abs{\varphi_r(t) - \varphi(t)}}{\abs{t}} \,\dint t\\
    &\qquad\leq C \int_{-T}^T \int_0^{s_0}  \abs{g_r(s) - g(s)}\, \mu(\dint s) \,\dint t
        + C \int_{-T}^T \abs{t} \int_{s_0}^\infty \abs{g_r(s)}^2 + \abs{g(s)}^2 \,\mu(\dint s)\, \dint t\\
    &\qquad\leq C T \int_0^{s_0}  \abs{g_r(s) - g(s)} \,\mu(\dint s)
        + C T^2 \int_{s_0}^\infty \abs{g_r(s)}^2 + \abs{g(s)}^2 \,\mu(\dint s)\\
    &\qquad\leq C T ( e^{-(k-1)r + (d-1)s_0} + h_k(r) e^{-2r + (d-k+2)s_0} ) + C T^2 e^{-(2k-d-1)s_0},
\end{aligned}\]
where \(h_3(r) = r\) and \(h_k(r)=1\) for \(k\geq 4\) (as in the proof of Lemma \ref{lem:bound_int}).
The Kolmogorov distance can thus be bounded from above by
\begin{equation}
\label{eq:bound_kol_2}
d_K(Y_r,Y) \leq C T ( e^{-(k-1)r + (d-1)s_0} + h_k(r) e^{-2r + (d-k+2)s_0} ) + C T^2 e^{-(2k-d-1)s_0} + C T^{-1}.
\end{equation}
It remains to find suitable values for \(T\) and \(s_0\) to balance out this expression.

We take the following approach:
Let \(\alpha, \beta >0\) be constants and set
\[ s_0 = \alpha r,\qquad T = e^{\beta r}. \]
Plugging this into \eqref{eq:bound_kol_2}, we obtain
\begin{equation*}\begin{aligned}
    d_K(Z_r, Z)
        &\leq C ( e^{\beta r - (k-1)r + (d-1)\alpha r} + h_k(r) e^{\beta r -2r + (d-k+2)\alpha r} ) + C e^{2\beta r-(2k-d-1)\alpha r} + C e^{-\beta r}\\
        &= C (e^{[\beta - (k-1) + (d-1)\alpha] r} + h_k(r) e^{[\beta  - 2 + (d-k+2)\alpha] r} + e^{[2\beta -(2k-d-1)\alpha] r} + e^{-\beta r}).
\end{aligned}\end{equation*}
To get an optimal rate of convergence, we  solve the following minimization problem:
\begin{equation*}
    \min_{\alpha,\beta>0}  \max \begin{Bmatrix*}[r]
        \beta &  + (d-1)\alpha & - (k-1),\\
        \beta & + (d-k+2)\alpha &- 2 ,\\
        2\beta & -(2k-d-1)\alpha,\\
        -\beta
    \end{Bmatrix*}.
\end{equation*}
This problem can be transformed into a linear optimization problem, which can then be solved using standard procedures.
The crucial first step is to rewrite the problem as
\[ \min t\colon \qquad \left\{\begin{array}{llrr}
    t &\geq \beta &  + (d-1)\alpha & - (k-1)\\
    t &\geq \beta & + (d-k+2)\alpha &- 2 \\
    t &\geq 2\beta & -(2k-d-1)\alpha& \\
    t &\geq -\beta & &
\end{array}\right\}  \quad\text{and}\quad \alpha, \beta > 0, \]
bringing this into standard form then is straightforward.

Since the problem is reasonably small, one can also solve it directly by hand.
If one chooses the second option, it is helpful to ignore the first term at the start.
Only optimizing the last three leads to
\[ \alpha^\ast = \frac{6}{k + d + 4},\qquad \beta^\ast = \frac{2(2k -d -1)}{k+d+4}, \]
and an optimum of \(-\beta^\ast\).
(At \((\alpha^\ast, \beta^\ast)\), the three terms coincide.
If one moves from \((\alpha^\ast,\beta^\ast)\) in any direction, then at least one of the terms increases, hence this pair yields the optimum.)
Plugging \(\alpha^\ast\) and \(\beta^\ast\) into the first term gives a value \(\leq -\beta^\ast\), hence one has obtained a solution of the full problem.

\subsection{Convergence of intersection processes}

For general \(m\), we show the following bound:

\begin{theorem}
\label{thm:rates_intersection}
    For \(2k > d+1\), let
    \[ \beta_{d,k} = \frac{2(2k -d -1)}{k+d+4}
    \qquad \text{and} \qquad
    w_{d,k}(r) = \begin{cases}
        e^{-\frac{2k-d-1}{3}r}&: \, 4k < 3d + 1,\\
        r^\frac{1}{3}e^{-\frac{2k-d-1}{3}r}&: \, 4k = 3d + 1,\\
        e^{-\frac{2(d-k)}{3}r}&: \, 4k > 3d + 1.
    \end{cases}\]
    If $k\ge 4$, then
    \begin{equation*}
        d_K\left(\frac{F^{(m)}_r - \E[F^{(m)}_r]}{e^{r(k-1)}}, C(d,k,m) \frac{\omega_k}{(k-1)2^{k-1}} Z\right)
            \leq C \cdot \exp(-\beta_{d,k}\cdot r) + C \cdot w_{d,k}(r),
    \end{equation*}
    for \(r \geq 1\) and if $k=3$, then
    \begin{equation*}
        d_K\left(\frac{F^{(m)}_r - \E[F^{(m)}_r]}{e^{r(k-1)}}, C(d,k,m) \frac{\omega_k}{(k-1)2^{k-1}} Z\right)
            \leq C \cdot r \cdot \exp(-\beta_{d,k}\cdot r) + C \cdot w_{d,k}(r),
    \end{equation*}
    for \(r \geq 1\), where \(C>0\) is some constant that depends only on \(d\) and \(k\).
\end{theorem}
As a direct consequence we get the following more concise bound:
\begin{corollary}\label{cor:rates}
   If \(\beta_{d,k}\) is as in Theorem \ref{thm:rates_intersection}, then
    \[\begin{aligned}
        &d_K\left(\frac{F^{(m)}_r - \E[F^{(m)}_r]}{e^{r(k-1)}}, C(d,k,m) \frac{\omega_k}{(k-1)2^{k-1}} Z\right)\\
        &\hspace{7cm}\leq \begin{cases}
            C \cdot r \cdot \exp(-\beta_{d,k} \cdot r) &: \, d=4,k=3, \\
            C \cdot \exp(- \frac{2}{3} \cdot  r) &: \, d \geq 12, k= d-1, \\
            C \cdot \exp(- \beta_{d,k} \cdot r) &: \, \text{otherwise},
        \end{cases}
    \end{aligned}\]
    for \(r\geq 1\), where \(C>0\) depends only on \(d\) and \(k\).
\end{corollary}
\begin{proof}
    Note that \(\exp(-\beta_{d,k}\cdot r) \lesssim w_{d,k}(r) = \exp(-\frac{2}{3}r) \) for \(d \geq 12\) and \(k = d-1\) and that
    in all other cases \(w_{d,k}(r) \lesssim \exp(-\beta_{d,k}\cdot r)\).
    The result then follows directly from \Cref{thm:rates_intersection}.
\end{proof}

We now prepare the proof of \Cref{thm:rates_intersection}.
Recall from \eqref{eq:decomp_F_m} that
\begin{align}
\nonumber
    \frac{F^{(m)}_r - \E[F^{(m)}_r]}{e^{r(k-1)}}
        &= C(d,k,m) \frac{F^{(1)}_r - \E[F^{(1)}_r]}{e^{r(k-1)}} + \sum_{i=2}^m \frac{I_i(f^{(m)}_{r,i})}{e^{r(k-1)}} \\
        &= C(d,k,m) Y_r + W_r, \label{eq:split}
\end{align}
with
\[ W_r \coloneqq \sum_{i=2}^m \frac{I_i(f^{(m)}_{r,i})}{e^{r(k-1)}}. \]
From \Cref{lem:A_asymptotics}, it follows that
\[ \V(W_r) \lesssim \frac{A^{(m)}_{r,2}}{e^{2r(k-1)}} \asymp \frac{1}{e^{2r(k-1)}} \cdot \begin{cases}
    e^{r(d-1)}&: \, 4(d-k)> d-1,\\
    re^{r(d-1)}&: \, 4(d-k)= d-1,\\
    e^{2r(d-2(d-k)-1)}&: \, 4(d-k)<d-1.
\end{cases} \]
Further recall that \(I_i(f^{(m)}_{r,i})\) is centered and hence \(\E[W_r] = 0\) and \(\V(W_r) = \E[W_r^2]\).
Using this fact and the above bound, we obtain
\begin{equation}
\label{eq:bound_W_squared}
    \E[W_r^2] \lesssim \begin{cases}
        e^{-(2k-d-1)r}&: \, 4k < 3d + 1,\\
        re^{-(2k-d-1)r}&: \, 4k = 3d + 1,\\
        e^{-2(d-k)r}&: \, 4k > 3d + 1.
    \end{cases}
\end{equation}
To bound the contribution of \(W_r\) to the Kolmogorov distance, we will use the following bound (for a related result, see \cite[Lem.~2.1]{KRT22+}):

\begin{lemma}
\label{lem:ineq_kol_dist}
    Let \(X, W, Z\) be random variables.
    Assume that \(Z\) has bounded density \(f_Z\).
    If \(\varepsilon > 0\), then
    \[ d_K(X + W, Z) \leq d_K(X, Z) + \varepsilon \norm{f_Z}_\infty + \P(\abs{W} > \varepsilon). \]
    If \(\E[W^2] < \infty\), it holds that
    \[ d_K(X + W, Z) \leq d_K(X, Z) + (\norm{f_Z}_\infty + 1) \E[W^2]^\frac{1}{3}. \]
\end{lemma}
\begin{proof}
    For \(\varepsilon > 0\) and \(t\in\R\), first observe that
    \[\begin{aligned}
        \P(X+W \leq t)
            &= \P(X+W \leq t, \abs{W}\leq \varepsilon) + \P(X+W\leq t, \abs{W}>\varepsilon)\\
            &\leq \P(X - \varepsilon \leq t) + \P(\abs{W}>\varepsilon)\\
            &= F_X(t+\varepsilon) + \P(\abs{W} > \varepsilon),
    \end{aligned}\]
    hence
    \begin{align*}
        &\P(X+W\leq t) - \P(Z\leq t)\\
            &\quad\leq F_X(t+\varepsilon) - F_Z(t) + \P(\abs{W} > \varepsilon)\\
            &\quad\leq \abs{F_X(t+\varepsilon) - F_Z(t+\varepsilon)} + \abs{F_Z(t+\varepsilon) - F_Z(t)} + \P(\abs{W} > \varepsilon)\\
            &\quad\leq d_K(X, Z) + \varepsilon \norm{f_Z}_\infty + \P(\abs{W} > \varepsilon).
    \end{align*}
    The bound in the other direction is achieved analogously.
    First, argue that
    \[ \P(X+W \leq t) \geq F_X(t-\varepsilon) - \P(\abs{W}>\varepsilon), \]
    then proceed as above.

    For the second statement, set \(\varepsilon \coloneqq \E[W^2]^\frac{1}{3}\).
    An application of the Markov inequality to the first bound gives
    \[ d_K(X + W, Z) \leq d_K(X, Z) + \varepsilon \norm{f_Z}_\infty + \frac{\E[W^2]}{\varepsilon^2}
        = d_K(X,Z) + (\norm{f_Z}_\infty + 1) \E[W^2]^\frac{1}{3}, \]
    which proves the remaining assertion.
\end{proof}

We now have all prerequisites to prove the main theorem of this section.

\begin{proof}[Proof of \Cref{thm:rates_intersection}]
    Note that \(C(d,k,m) \frac{\omega_k}{(k-1)2^{k-1}} Z\) has bounded density by \Cref{lem:bounded_density}.
    We can thus apply \Cref{lem:ineq_kol_dist} to \eqref{eq:split}, which yields
    \[\begin{aligned}
        &d_K\left(\frac{F^{(m)}_r - \E[F^{(m)}_r]}{e^{r(k-1)}}, C(d,k,m) \frac{\omega_k}{(k-1)2^{k-1}} Z\right)\\
        &\qquad\leq d_K\left(C(d,k,m)\cdot Y_r, C(d,k,m) \frac{\omega_k}{(k-1)2^{k-1}} Z\right) + C \cdot \E[W_r^2]^\frac{1}{3}
    \end{aligned}\]
    for some constant \(C>0\) that depends on \(d\) and \(k\).
    (Technically, \(C\) depends on \(m\) as well, but we can get rid of this dependence by taking the maximum over all admissible \(m\).)
    Using the fact that \(d_K(aU, aV) = d_K(U,V)\), for any random variables \(U, V\) and any constant \(a \neq 0\), and writing
    \(Y = \frac{\omega_k}{(k-1)2^{k-1}}Z\), we get
    \[d_K\left(\frac{F^{(m)}_r - \E[F^{(m)}_r]}{e^{r(k-1)}}, C(d,k,m) \frac{\omega_k}{(k-1)2^{k-1}} Z\right)
        \leq d_K\left(Y_r, Y\right) + C \cdot \E[W_r^2]^\frac{1}{3}.\]
    The claim now follows from \Cref{thm:convergence_rates_1} and \eqref{eq:bound_W_squared}.
\end{proof}

    \section{Numerical calculations}
\label{sec:density_approximation}

To better understand the limit distribution for $2k>d+1$, we would like to calculate its probability density.
For this purpose, write \(\psi_{d,k}\) for the characteristic function of \(Z = Z_{d,k}\), that is
\[ \psi_{d,k}(t) = \exp\left( \omega_{d-k} \int_0^\infty (e^{it h(s)} - 1 - it h(s)) \,\mu(\dint s) \right),\qquad t\in\R, \]
where \(h(s) = \cosh^{-(k-1)}(s)\) for $s\in\R$.

From the characteristic function, we can determine the cumulants and thus the moments of \(Z_{d,k}\) (see \cite[Rem.~4.2]{BHT}),
in particular we see that the mean is zero and that the variance is given by
\[ \sigma_{d,k}^2 \coloneqq \V(Z_{d,k}) = \omega_{d-k} \int_0^\infty \cosh^{2-k}(s) \sinh^{d-k-1}(s) \dint s. \]
To compare the limit distributions for different values of \(d\) and \(k\), it is useful to consider the standardized random variable
\(Z_{d,k}^* \coloneqq Z_{d,k}/\sigma_{d,k}\).
The corresponding characteristic function \(\widetilde{\psi}_{d,k}\) is given by 
\[ \widetilde{\psi}_{d,k}(t) = \psi_{d,k}(t/\sigma_{d,k}),\qquad t\in\R. \]

The density \(f_{d,k}\) of \(Z_{d,k}^*\) can be obtained from the classical  inversion formula (see \cite[Thm.~3.3.14]{Durrett})
\begin{equation}
\label{eq:inversion_formula}
    f_{d,k}(x) = \frac{1}{2\pi} \int_{-\infty}^\infty e^{-itx} \widetilde{\psi}_{d,k}(t) \dint t,\qquad x\in\R.
\end{equation}
The integrability of the characteristic function and thus the applicability of the inversion formula was already established
in \Cref{lem:bounded_density}.
Since it does not seem to be possible to express this integral by elementary functions, we will approximate the right-hand side of  \eqref{eq:inversion_formula} numerically.

From the considerations in \Cref{lem:bounded_density} it follows that \(\widetilde{\psi}_{d,k}\) decays exponentially in \(t\),
so it makes sense to approximate the integral on the right-hand side of \eqref{eq:inversion_formula} by integrating over a bounded interval:
\[ f_{d,k}(x) \approx \frac{1}{2\pi} \int_{-T}^{T} e^{-itx} \widetilde{\psi}_{d,k}(t) \dint t.  \]
To discretize the integral, we take \(M+1 \in \N\) equidistant points \(t_j = \left(-1+\frac{2j}{M}\right)T\), \(j=0,\ldots,M\), and
use the left point approximation \(\int_a^b g(t) \dint t \approx (b-a) g(a)\):
\[ \frac{1}{2\pi} \int_{-T}^T e^{-itx} \widetilde{\psi}_{d,k}(t) \dint t
        = \frac{1}{2\pi} \sum_{j=0}^{M-1} \int_{t_j}^{t_{j+1}} e^{-itx} \widetilde{\psi}_{d,k}(t) \dint t
        \approx \frac{T}{\pi  M} \sum_{j=0}^{M-1} e^{-it_{j}x} \widetilde{\psi}_{d,k}(t_j).  \]

The evaluation of \(\widetilde{\psi}_{d,k}(t)\) comes with some difficulties. 
The characteristic function itself contains an integral that needs to be numerically approximated first.
In practice this turned out to be challenging for the computer, so we took a different approach.

To approximate
\[ \widetilde{\psi}_{d,k}(t)
    = \exp\left( \omega_{d-k} \int_0^\infty (e^{it h(s)/\sigma_{d,k}} - 1 - it h(s)/\sigma_{d,k}) \,\mu(\dint s) \right),\]
we expand  the exponent and interchange summation and integration to get
\[\begin{aligned}
    &\omega_{d-k} \int_0^\infty \sum_{n=2}^\infty \frac{(ith(s)/\sigma_{d,k})^n}{n!} \,\mu(\dint s)\\
    &\qquad= \omega_{d-k} \int_0^\infty \sum_{n=2}^\infty \frac{(it/\sigma_{d,k})^n}{n!} \cosh^{k-n(k-1)}(s) \sinh^{d-k-1}(s)\, \dint s\\
    &\qquad= \omega_{d-k} \sum_{n=2}^\infty \frac{(it/\sigma_{d,k})^n}{n!} \int_0^\infty \cosh^{k-n(k-1)}(s) \sinh^{d-k-1}(s) \,\dint s
\end{aligned}\]
 for the argument of the exponential function. 
The application of the dominated convergence theorem in the last equality is justified since \(2k > d+1\) implies
\[\begin{aligned}
    \int_0^\infty \cosh^{k-n(k-1)}(s) \sinh^{d-k-1}(s) \,\dint s
    &\leq \int_0^\infty \left(e^s/2\right)^{(k-n(k-1))} \left(e^s/2\right)^{(d-k-1)}\, \dint s\\
    &\leq 2^{n(k-1)-d+1} \int_0^\infty e^{-s} \,\dint s \leq \left(2^{k-1}\right)^n
\end{aligned}\]
for all \(n\geq 2\).

As pointed out in \cite[Rem.~4.2]{BHT}, for \(-a > b > -1\) it holds that
\begin{equation*}
    I(a,b) \coloneqq \int_0^\infty \cosh^a(s) \sinh^b(s) \dint s = \frac{1}{2}\frac{\Gamma(-\frac{a+b}{2}) \Gamma(\frac{b+1}{2})}{\Gamma(\frac{1-a}{2})}.
\end{equation*}
With this identity and the definition of \(\sigma_{d,k}\), we can rewrite the above expression as
\[\begin{aligned}
    &\omega_{d-k} \sum_{n=2}^\infty \frac{(it)^n}{n!} (\omega_{d-k} I(2-k, d-k-1))^{-\frac{n}{2}} I(k-n(k-1), d-k-1)\\
    &\qquad= \sum_{n=2}^\infty \frac{(it)^n}{n!} \omega_{d-k}^{1-\frac{n}{2}}
        \left( \frac{1}{2}\frac{\Gamma(\frac{2k-d-1}{2}) \Gamma(\frac{d-k}{2})}{\Gamma(\frac{k-1}{2})} \right)^{-\frac{n}{2}}
        \frac{1}{2}\frac{\Gamma(\frac{n(k-1)+1-d}{2}) \Gamma(\frac{d-k}{2})}{\Gamma(\frac{n(k-1)+1-k}{2})}\\
    &\qquad= \sum_{n=2}^\infty \frac{(it)^n}{n!} \left(\frac{1}{2} \Gamma\left(\frac{d-k}{2}\right) \omega_{d-k}\right)^{1-\frac{n}{2}}
        \left( \frac{\Gamma(\frac{2k-d-1}{2})}{\Gamma(\frac{k-1}{2})} \right)^{-\frac{n}{2}}
        \frac{\Gamma(\frac{n(k-1)+1-d}{2})}{\Gamma(\frac{n(k-1)+1-k}{2})}\\
    &\qquad= \sum_{n=2}^\infty \frac{(it)^n}{n!} \left( \pi^{\frac{d-k}{2}} \right)^{1-\frac{n}{2}}
        \left( \frac{\Gamma(\frac{k-1}{2})}{\Gamma(\frac{2k-d-1}{2})} \right)^{\frac{n}{2}}
        \frac{\Gamma(\frac{n(k-1)+1-d}{2})}{\Gamma(\frac{n(k-1)+1-k}{2})}.
\end{aligned}\]
We then approximate this sum by then \(N\)-th partial sum for some integer \(N \geq 2\).
Note that for higher values of \(t\) one might have to choose \(N\) larger to get a sufficient accuracy.

From the above calculation it also follows that for \(n\geq 2\) the cumulants of \(Z_{d,k}^*\) can be expressed as
\begin{align}
    \cum_n(Z_{d,k}^*)
        &= (-i)^n \frac{\partial^n}{\partial t^n} \log \widetilde{\psi}_{d,k}(t) \big\vert_{t=0} \nonumber\\
        &= \left( \pi^{\frac{d-k}{2}} \right)^{1-\frac{n}{2}}
            \left( \frac{\Gamma(\frac{k-1}{2})}{\Gamma(\frac{2k-d-1}{2})} \right)^{\frac{n}{2}}
            \frac{\Gamma(\frac{n(k-1)+1-d}{2})}{\Gamma(\frac{n(k-1)+1-k}{2})},
    \label{eq:cumulants}
\end{align}
which will be useful later.

There are three constants that need to be chosen for the approximation, namely \(T, M\) and \(N\).
Finding appropriate values can be achieved by a trial and error approach.
We only calculated characteristic functions and densities for \(d \leq 15\), since after that point the calculations could not
be carried out accurately enough by the computer anymore.
For those values of \(d\) (and \(\frac{d+1}{2} < k \leq d-1\)) the choices \(T = 10\), \(M = 200\) and \(N = 26\) have proven to be appropriate.

\begin{figure}
    \centering
    \begin{subfigure}{0.45\textwidth}
        \includegraphics[width=\textwidth]{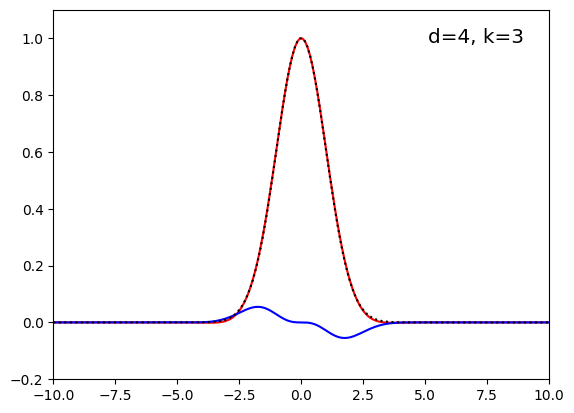}
    \end{subfigure}
    \hfill
    \begin{subfigure}{0.45\textwidth}
        \includegraphics[width=\textwidth]{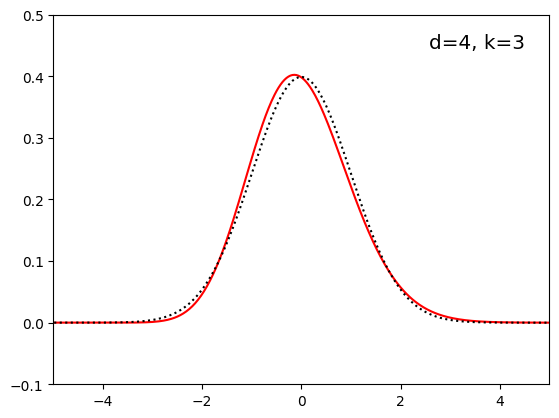}
    \end{subfigure}
    \begin{subfigure}{0.45\textwidth}
        \includegraphics[width=\textwidth]{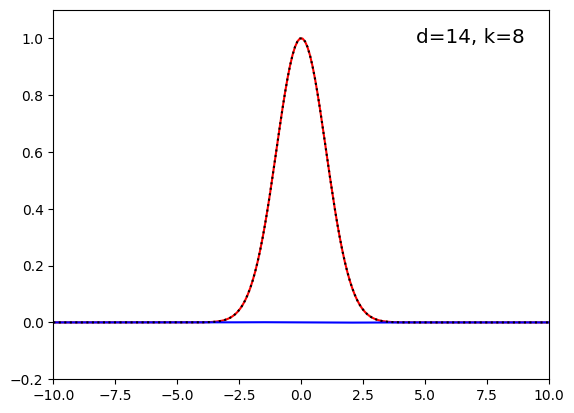}
    \end{subfigure}
    \hfill
    \begin{subfigure}{0.45\textwidth}
        \includegraphics[width=\textwidth]{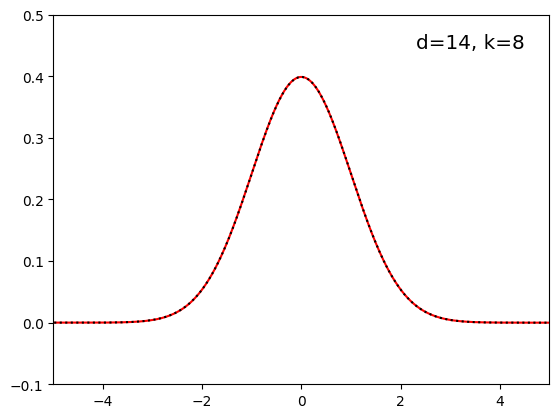}
    \end{subfigure}
    
    \begin{subfigure}{0.45\textwidth}
        \includegraphics[width=\textwidth]{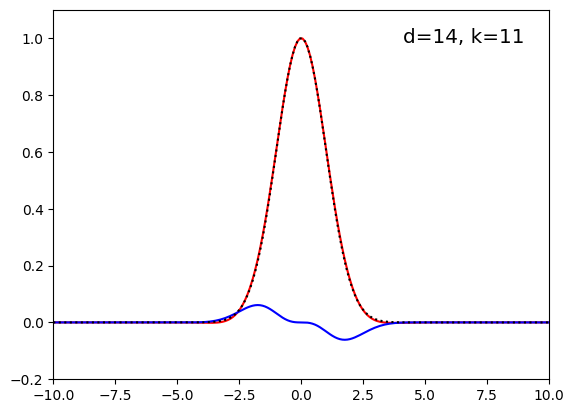}
    \end{subfigure}
    \hfill
    \begin{subfigure}{0.45\textwidth}
        \includegraphics[width=\textwidth]{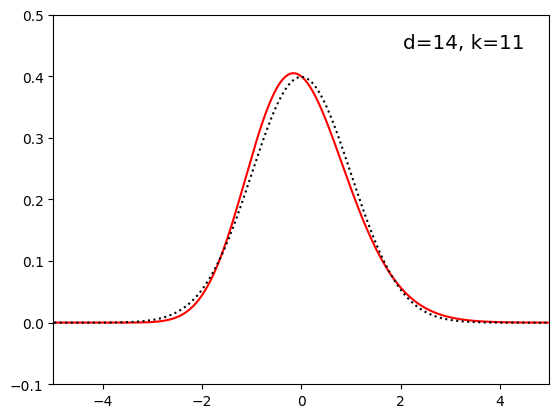}
    \end{subfigure}
    \begin{subfigure}{0.45\textwidth}
        \includegraphics[width=\textwidth]{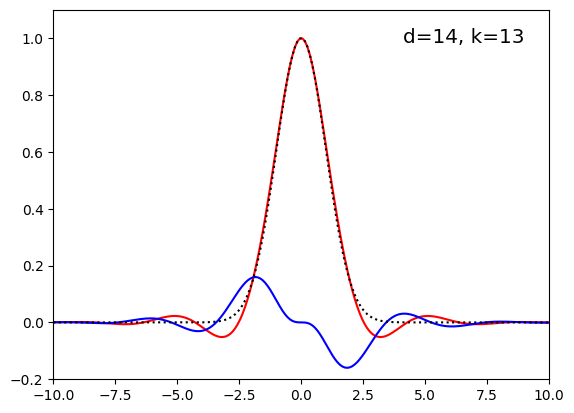}
    \end{subfigure}
    \hfill
    \begin{subfigure}{0.45\textwidth}
        \includegraphics[width=\textwidth]{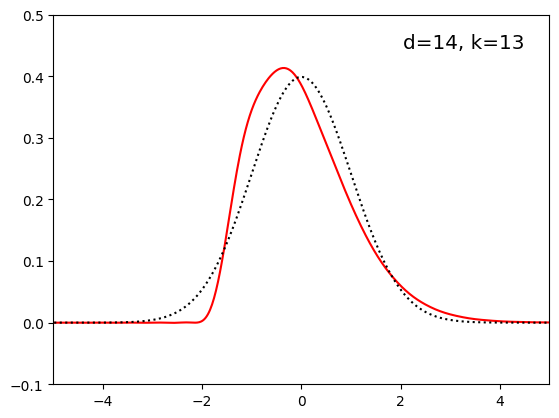}
    \end{subfigure}
    \caption{Left: Real part (red) and imaginary part (blue) of \(\widetilde{\psi}_{d,k}\).
    The dotted line is the (entirely real) characteristic function of a standard normal distribution.
    Right: Density of \(Z_{d,k}^*\) (red) and density of a standard normal distribution (dotted).}
    \label{fig:figure_1}
    \vspace{-27pt}
\end{figure}

The illustrations in \Cref{fig:figure_1} might lead one to the following conjecture:
If one lets \(d\) go to infinity and takes \(k \approx d/2\), then the distribution will converge to a standard normal distribution, while
for \(k\approx d\), it will diverge.
The behavior of the cumulants, on the other hand, does not seem to support this hypothesis. 
It is easy to see that the first cumulant is \(0\) and since \(Z_{d,k}/\sigma_{d,k}\) has variance \(1\), the second cumulant is equal to \(1\),
so the first two cumulants always agree with the ones of a standard normal distribution.
However, for \(n \geq 3\), the \(n\)-th cumulant of a standard normal distribution is zero, while the ones of \(Z_{d,k}/\sigma_{d,k}\)
turn out to diverge as \(d\) approaches infinity, independently of how \(k\) is chosen.

\begin{proposition}
    Let \((d_j), (k_j)\) be sequences of integers that satisfy \(0 \leq k_j \leq d_j -1\),  \( 2 k_j > d_j + 1 \)
    and \(d_j \to \infty\), \(j\to \infty\).
    It then holds that as \(j \to \infty\),
    \[ \cum_n(Z_{d_j,k_j}^*) \to \infty, \]
    for all \(n \geq 3\).
\end{proposition}
\begin{proof}
    Fix \( n \geq 3\).
    For the sake of notational convenience we will sometimes drop the index \(j\).

    If \(d_j-k_j = 2c\) is constant, we can use the fact that \(\Gamma(x+y)/\Gamma(x) \sim x^y\) for fixed \(y > 0\) and \(x \to \infty\)
    (see, e.g., \cite[p.~938]{jameson}) and  \eqref{eq:cumulants} to conclude that as \(j \to \infty\)
    \begin{align*}
        \cum_n(Z_{d,k}^*)
            &= \left( \pi^{\frac{d-k}{2}} \right)^{1-\frac{n}{2}}
                \left( \frac{\Gamma(\frac{2k-d-1}{2}+c)}{\Gamma(\frac{2k-d-1}{2})} \right)^{\frac{n}{2}}
                \frac{\Gamma(\frac{n(k-1)+1-d}{2})}{\Gamma(\frac{n(k-1)+1-d}{2}+c)}\\
            &\sim \left( \pi^{c} \right)^{1-\frac{n}{2}}
                ( \tfrac{2k-d-1}{2} )^{c\frac{n}{2}}
                ( \tfrac{n(k-1)+1-d}{2} )^{-c}\\
            &= \left( \pi^{c} \right)^{1-\frac{n}{2}}
                \left( \frac{\left(\frac{d-4c-1}{2}\right)^\frac{n}{2}}{\left(\frac{n(d-2c-1)+1-d}{2}\right)} \right)^c,
    \end{align*}
    which diverges since \(n \geq 3\) and \(d_j\to \infty\).  In particular, this proves the assertion if $d_j-k_j\in\{1,\ldots,5\}$. In the following we can therefore assume that $d_j- k_j\ge 6$. 

    The Gamma function satisfies \(\Gamma(x+1) = x \Gamma(x)\) for \(x > 0\) and is increasing on \([2,\infty)\).
   We can thus bound the cumulants given in \eqref{eq:cumulants} from below by
    \begin{align}
    \cum_n(Z_{d,k}^*)
        &\geq \left( \pi^{\frac{d-k}{2}} \right)^{1-\frac{n}{2}}
            \left( \frac{\Gamma(\frac{2k-d-1}{2}+\lfloor \frac{d-k}{2} \rfloor)}{\Gamma(\frac{2k-d-1}{2})} \right)^{\frac{n}{2}}
            \frac{\Gamma(\frac{n(k-1)+1-d}{2})}{\Gamma(\frac{n(k-1)+1-d}{2}+ \lceil \frac{d-k}{2} \rceil)}\nonumber\\
        &= \left( \pi^{\frac{d-k}{2}} \right)^{1-\frac{n}{2}}
            \prod_{j=0}^{\lfloor\frac{d-k}{2}\rfloor -1} \left( \tfrac{2k-d-1}{2}+ j \right)^\frac{n}{2} 
            \prod_{j=0}^{\lceil\frac{d-k}{2}\rceil -1}
            \frac{1}{\left(\frac{n(k-1)+1-d}{2}+j\right)}\nonumber\\
        &\geq \left( \pi^{\frac{d-k}{2}} \right)^{1-\frac{n}{2}}
            \prod_{j=0}^{\lfloor\frac{d-k}{2}\rfloor -1} \frac{\left( \frac{2k-d-1}{2}+ j \right)^\frac{n}{2}}{\left(\frac{n(k-1)+1-d}{2}+j\right)}
            \cdot \frac{1}{\left(\frac{n(k-1)+1-k}{2}\right)}.
            \label{eq:cumulants_2}
    \end{align}
    
    For \( k_j\ge \frac{3}{4}d_j\) and \(d_j \geq 4\) we can then plug the bound
    \[ \prod_{j=0}^{\lfloor\frac{d-k}{2}\rfloor -1} \frac{\left( \frac{2k-d-1}{2}+ j \right)^\frac{n}{2}}{\left(\frac{n(k-1)+1-d}{2}+j\right)}
    \geq \prod_{j=0}^{\lfloor\frac{d-k}{2}\rfloor -1} \frac{\left( \frac{2k-d-1}{2} \right)^\frac{n}{2}}{\left(\frac{n(k-1)+1-k}{2}\right)}
    \geq \left( \frac{(\frac{1}{8}d)^\frac{n}{2}}{nd} \right)^{\lfloor\frac{d-k}{2}\rfloor}
    = \left( \frac{d^{\frac{n}{2}-1}}{8^\frac{n}{2}n}  \right)^{\lfloor\frac{d-k}{2}\rfloor} \]
   into \eqref{eq:cumulants_2} and use \(d^{\frac{n}{2}-1} \geq d^\frac{1}{2}\) to get
    \[ \cum_n(Z_{d,k}^*)
    \geq \left( \pi^{1-\frac{n}{2}} \frac{d^\frac{1}{2}}{8^\frac{n}{2}n} \right)^{\lfloor\frac{d-k}{2}\rfloor}
    \cdot \frac{\pi^{1-\frac{n}{2}}}{\left(\frac{(n-1)(k-1)}{2}\right)}
    \geq \left( \pi^{1-\frac{n}{2}} \frac{d^\frac{1}{2}}{8^\frac{n}{2}n} \right)^{\lfloor\frac{d-k}{2}\rfloor}
    \cdot \frac{\pi^{1-\frac{n}{2}}}{nd} , \]
    which goes to infinity as \(d_j \to \infty\) since \(\lfloor\frac{d-k}{2}\rfloor \geq 3\).

    For \(k_j \leq \frac{3}{4}d_j\) and \(d_j \geq 16\), we use \eqref{eq:cumulants_2} to derive the lower bound
    \begin{align*}
        \cum(Z_{d,k}^*)
        &\geq \left( \pi^{\frac{d-k}{2}} \right)^{1-\frac{n}{2}}
            \prod_{j=0}^{\lfloor\frac{d-k}{2}\rfloor -1} \frac{\left( \frac{1}{2}+ j \right)^\frac{n}{2}}{\left(\frac{n(k-1)+1-k}{2}\right)}
            \cdot \frac{1}{\left(\frac{n(k-1)+1-k}{2}\right)}\\
        &\geq \left( \pi^{\frac{d-k}{2}} \right)^{1-\frac{n}{2}}
             \frac{\left( \frac{1}{2} \left(\lfloor\frac{d-k}{2}\rfloor -1\right)! \right)^\frac{n}{2}}
             {\left(\frac{(n-1)(k-1) }{2}\right)^{\lfloor\frac{d-k}{2}\rfloor +1}}\\
        &\geq \left( \pi^{\frac{d-k}{2}} \right)^{1-\frac{n}{2}} 2^{-\frac{n}{2}}
            \frac{\left(\frac{\lfloor\frac{d-k}{2}\rfloor -1}{e}\right)^{(\lfloor\frac{d-k}{2}\rfloor -1)\frac{n}{2}}}
            {\left(nd\right)^{\lfloor\frac{d-k}{2}\rfloor +1}}\\
        &\geq \left( \pi^{1-\frac{n}{2}} \frac{\left(\frac{\frac{d}{8}-2}{e} \right)^\frac{n}{2}} {nd} \right)^{\lfloor\frac{d-k}{2}\rfloor -1}
            2^{-\frac{n}{2}} \pi^{2(1-\frac{n}{2})} \frac{1}{(nd)^2},
    \end{align*}
    where we used the fact that \(m! \geq (m/e)^m\) for \(m \in \N\) in the third line.
    For sufficiently large \(d_j\), the above expression is lower bounded by
    \[ \left(A\cdot d^\frac{1}{2}\right)^{\lfloor \frac{d}{8} \rfloor -1} \cdot \frac{B}{d^2} \]
    with some constants \(A,B > 0\) (depending on $n$).
    The latter expression goes to infinity as \(d_j \to \infty\).
\end{proof}

Of course, this does not yet prove that \(Z_{d,k}^*\) does not converge in distribution to a normal distribution (or at all).
It is indeed possible for a sequence of functions to converge even though, when written as a power series, all of the coefficients diverge (except for the coefficient of $x^0$).
This is the case, e.g., for \(f_n(x) = \sin(n^2 x)/n + \cos(n^2 x)/n\), $x\in\R$.
Unfortunately, we were not able to determine whether the characteristic functions  converge or not and leave this question as an open problem.

\begin{question*}
    Let \((d_j), (k_j)\) be sequences of integers satisfying \(0 \leq k_j \leq d_j -1\),  \( 2 k_j > d_j + 1 \)
    and \(d_j \to \infty\), \(j\to \infty\).
    What is the asymptotic behavior of \(Z_{d_j,k_j}^*\) as \(j\to\infty\)?
    In particular, are there any conditions on \(d_j\) and \(k_j\) (such as \(k_j / d_j \to 1/2\)) which ensure that \(Z_{d_j,k_j}^*\) converges in distribution, and if so, is the limit distribution Gaussian?
\end{question*}
	\section{Asymptotic covariance matrix}
\label{sec:covariance_matrix}

Consider the random vector \(F^{[m]}_r \coloneqq (F^{(1)}_r, \ldots, F^{(m)}_r)^\top \).
In this section we study the asymptotic covariance matrix \(\Sigma^{[m]}\) of \(F^{[m]}_r\), which we define as
\[ \lim_{r \to \infty} \frac{1}{e^{r(d-1)}} \Sigma(F^{[m]}_r),
    \quad \lim_{r \to \infty} \frac{1}{r e^{r(d-1)}} \Sigma(F^{[m]}_r),
    \quad \lim_{r \to \infty} \frac{1}{e^{2r(k-1)}} \Sigma(F^{[m]}_r), \]
for \(2k<d+1\), \(2k=d+1\) and \(2k>d+1\),  respectively.

\begin{proposition}
\label{prop:covariance_matrix}
    For \(m \geq 2\) the following holds:
    If \(2k \geq d+1\), then \(\Sigma^{[m]}\) has rank one, while it has full rank when \(2k< d+1\).
\end{proposition}

\begin{remark}
    Note that \(m\) must satisfy the condition \(d - m(d-k) \geq 0\).
    One easily checks that the only combinations for \(m, d, k\) that satisfy this condition as well as \(m \geq 2\) and \(2k<d+1\) are those where \(m=2\) and \(d=2k\).
\end{remark}

In the remainder of this section we prove \Cref{prop:covariance_matrix} and derive \(\Sigma^{[2]}\) for \(2k<d+1\).
For \(2k \geq d+1\), note that by \eqref{eq:decomp_F_m}
\[F^{(m)}_r - \E[F^{(m)}_r] = C(d,k,m) \cdot (F^{(1)}_r - \E[F^{(1)}_r]) + \textrm{lower order term}, \]
hence we can argue as in \Cref{sec:comparison} that the asymptotic covariance matrix has rank one.
This proves the first part of \Cref{prop:covariance_matrix}.

Let us from now on assume that \(2k < d+1\).
By the above remark, we can assume that \(m=2\) and \(d=2k\).
We exclude \(k=1\) in the derivation below, since this case has already been treated in \cite[Sec.~4.5.1]{HHT}.
From the Wiener--Itô chaos decomposition and \eqref{eq:crofton_applied} it follows that
\[\begin{aligned}
    F^{(1)}_r &= \E[F^{(1)}_r] + I_1(f^{(1)}_{r,1}),\\
    F^{(2)}_r &= \E[F^{(2)}_r] + C(d,k,2) \cdot I_1(f^{(1)}_{r,1}) + I_2(f^{(2)}_{r,2}).
\end{aligned}\]
Since the first and second chaos term are uncorrelated, the covariance matrix \(\Sigma(F^{[m]})\)  takes the form
\[\begin{aligned}
    \Sigma\left(F^{[2]}_r\right)
        &= \Sigma\left(I_1(f^{(1)}_{r,1}), C(d,k,2) \cdot I_1(f^{(1)}_{r,1})\right) + \Sigma\left(0, I_2(f^{(2)}_{r,2})\right)\\
        &= \V(I_1(f^{(1)}_{r,1})) \begin{pmatrix} 1 & C(d,k,2)\\ C(d,k,2) & C(d,k,2)^2  \end{pmatrix} + \V(I_2(f^{(2)}_{r,2})) \begin{pmatrix} 0&0\\ 0&1 \end{pmatrix}.
\end{aligned} \]

Note that, by definition, \(\V(I_1(f^{(1)}_{r,1})) = A^{(1)}_{r,1}\) and \(\V(I_2(f^{(2)}_{r,2})) = 2 A^{(2)}_{r,2} \).
Using the representation \cite[Eq.~(2.9)]{BHT}  and writing the involved integral as in \cite[Lem.~8 and (19)]{HHT}, we arrive at
\[ A^{(1)}_{r,1} = c^{(1)}_{d,k} \int_0^r \cosh^k(s) \sinh^{d-1-k}(s) \left( \int_0^{\arcosh\left( \frac{\cosh(r)}{\cosh(s)} \right)} \sinh^{k-1}(u) \dint u \right)^2 \dint s, \]
\[ A^{(2)}_{r,2} = c^{(2)}_{d,k} \int_0^r \sinh^{d-1}(s) \dint s, \]
with constants \(c^{(1)}_{d,k} = \omega_k^3\) and \(c^{(2)}_{d,k} = \frac{\omega_1 \omega_d \omega_{d+1}}{4 \omega_{k+1}^2}\).

After substituting \(t = r - s\), we get that
\[ e^{-(d-1)r} A^{(1)}_{r,1} = c^{(1)}_{d,k} \int_0^r \frac{\cosh^k(r-t) \sinh^{d-1-k}(r-t)}{e^{(d-1)r}} \left( \int_0^{\arcosh\left( \frac{\cosh(r)}{\cosh(r-t)} \right)} \sinh^{k-1}(u) \dint u \right)^2 \dint t \]
To bound the integrand, first note that
\[ \frac{\cosh^k(r-t) \sinh^{d-1-k}(r-t)}{e^{(d-1)r}} \leq \frac{e^{(d-1)(r-t)}}{e^{(d-1)r}} = e^{-(d-1)t}, \]
and since
\[ \frac{\cosh(r)}{\cosh(r-t)} = e^t \frac{1 + e^{-2r}}{1 + e^{2t-2r}} \leq e^t \]
and \(\exp(\arcosh(y)) \leq 2y\) for \(y\geq 1\), we can bound
\[\begin{aligned}
    \int_0^{\arcosh\left( \frac{\cosh(r)}{\cosh(r-t)} \right)} \sinh^{k-1}(u) \dint u
        &\leq \int_0^{\arcosh( e^t )} e^{(k-1)u} \dint u\\
        &=\frac{1}{k-1} \left( \exp(\arcosh(e^t))^{k-1} - 1\right)\\
        &\leq \frac{2^{k-1}}{k-1} e^{(k-1)t}.
\end{aligned}\]
The (non-negative) integrand can thus be bounded uniformly in \(r\) from above by
\[ e^{-(d-1)t} \left(\frac{2^{k-1}}{k-1}\right)^2 e^{2(k-1)t} = \left(\frac{2^{k-1}}{k-1}\right)^2 e^{-t}, \]
where we used that \(2k=d\).
We can thus apply the dominated convergence theorem and obtain
\[ \lim_{r\to\infty} e^{-(d-1)r} A^{(1)}_{r,1}
    = c^{(1)}_{d,k} \int_0^\infty \frac{e^{-(d-1)t}}{2^{d-1}} \left( \int_0^{\arcosh(e^t)} \sinh^{k-1}(u) \dint u \right)^2 \dint t. \]
After substituting \(u = \arcosh(y)\), the inner integral becomes
\[\int_1^{e^t} (y^2 - 1)^\frac{k-2}{2} \dint y,\]
the substitution \(x = e^t\) then turns the outer integral into
\[ \int_1^\infty \frac{x^{-d}}{2^{d-1}} \left( \int_1^{x} (y^2 - 1)^\frac{k-2}{2} \dint y \right)^2 \dint x, \]
so that
\[ \lim_{r\to\infty} e^{-(d-1)r} A^{(1)}_{r,1}
    = \frac{c^{(1)}_{d,k}}{2^{d-1}} \int_1^\infty x^{-d} \left( \int_1^{x} (y^2 - 1)^\frac{k-2}{2} \dint y \right)^2 \dint x. \]
(We excluded \(k=1\) for convenience, but one can check that the final result is also valid in this case.)
The second term is much easier.
One again substitutes \(t = r - s\), checks that the dominated convergence theorem applies and obtains
\[ \lim_{r\to\infty} e^{-(d-1)r} A^{(2)}_{r,2}
    = \frac{c^{(2)}_{d,k}}{2^{d-1}} \int_0^\infty e^{-(d-1)t} \dint t
    = \frac{c^{(2)}_{d,k}}{(d-1)2^{d-1}}.  \]
In summary, we arrive at
\begin{align}
&\Sigma(F^{[2]}) = \nonumber\\
&\quad\frac{\omega_k^3}{2^{d-1}} \int_1^\infty x^{-d} \left( \int_1^{x} (y^2 - 1)^\frac{k-2}{2} \dint y \right)^2 \dint x
\begin{pmatrix}
1 & \frac{\omega_1 \omega_{d+1}}{\omega_{k+1}^2}\\
\frac{\omega_1 \omega_{d+1}}{\omega_{k+1}^2} & \left( \frac{\omega_1 \omega_{d+1}}{\omega_{k+1}^2} \right)^2
\end{pmatrix}
+ \begin{pmatrix} 0 & 0\\ 0 &  \frac{\omega_1 \omega_d \omega_{d+1}}{(d-1) 2^{d} \omega_{k+1}^2 } \end{pmatrix}. \label{eq:covariance_matrix}
\end{align}
From this representation the second part of \Cref{prop:covariance_matrix} follows immediately.

\begin{remark}
    The integral in \eqref{eq:covariance_matrix} can be explicitly calculated when \(k\) is even, since the inner integrand is then only a polynomial.
    For odd \(k\) this is in general much harder.
    In \cite[Rem.~11]{HHT} it has been shown that for \(k=1,d=2\) this integral is equal to \(a = 4G\), where 
    \(G \approx 0.915965594\) is Catalan's constant.
    The covariance matrix they arrive at is
    \[\begin{pmatrix}
        4a & \frac{8}{\pi} a\\
        \frac{8}{\pi} a & \frac{16}{\pi^2}a + \frac{1}{\pi}
    \end{pmatrix},\]
    compare \cite[Eq.~(20)]{HHT} (we set the intensity parameter \(t\) to one and switched the order of the rows and columns to conform to our setting and notation).
    Plugging the values \(k=1,d=2\) into our representation \eqref{eq:covariance_matrix} we arrive at
    \[\begin{pmatrix}
        4a & \frac{8}{\pi} a\\
        \frac{8}{\pi} a & \frac{16}{\pi^2}a + 1
    \end{pmatrix}.\]
    The discrepancy is resolved by inserting a missing factor $\pi$ in front of the integral on page 925, line -8, in \cite{HHT} 
    (\(\omega_1\) was used instead of the correct factor \(\omega_2\)).
\end{remark}

\subsection*{Acknowledgements}
The authors have been funded by the German Research Foundation (DFG) through the Priority Programme ``Random Geometric Systems''  (SPP 2265), via the research grant HU 1874/5-1.

	\printbibliography
\end{document}